\documentclass[12pt,reqno]{amsart}
\usepackage{amsmath, amsfonts, amsthm, amssymb, mathtools, bm, longtable, booktabs,diagbox}
\usepackage[draft]{hyperref}
\usepackage{geometry}
\usepackage{verbatim}
\usepackage{enumitem}

\theoremstyle{plain}
\newtheorem{theorem}{Theorem}[section]
\newtheorem{corollary}[theorem]{Corollary}
\newtheorem{conjecture}[theorem]{Conjecture}
\newtheorem{lemma}[theorem]{Lemma}
\newtheorem{proposition}[theorem]{Proposition}
\theoremstyle{definition}

\theoremstyle{remark}
\newtheorem*{remark}{Remark}

\numberwithin{equation}{section}

\newcommand{\N}{\mathbb N}
\newcommand{\Z}{\mathbb Z}
\newcommand{\C}{\mathbb C}

\newcommand{\rank}{\operatorname{rank}}
\newcommand{\crank}{\operatorname{crank}}

\newcommand{\Arg}{\operatorname{Arg}}
\newcommand{\Log}{\operatorname{Log}}

\begin{document}
\allowdisplaybreaks

%\subjclass[2010]{}
%\keywords{}

\begin{abstract}

We study the asymptotic behavior of the rank statistic for unimodal sequences. We use analytic techniques involving asymptotic expansions in order to prove asymptotic formulas for the moments of the rank. Furthermore, when appropriately normalized, the values of the unimodal rank asymptotically follow a logistic distribution. We also prove similar results for Durfee unimodal sequences and semi-strict unimodal sequences, with the only major difference being that the (normalized) rank for semistrict unimodal sequences has a distributional limit of a point mass probability distribution.

\end{abstract}

\title{The asymptotic distribution of the rank for unimodal sequences}
\author[K. Bringmann]{Kathrin Bringmann}
\address{University of Cologne, Faculty of Mathematical and Natural Sciences, Mathematical Institute, Weyertal 86-90, 50931 Cologne, Germany}
\email{kbringma@math.uni-koeln.de}

\author[C. Jennings-Shaffer]{Chris  Jennings-Shaffer}
\address{University of Cologne, Faculty of Mathematical and Natural Sciences, Mathematical Institute, Weyertal 86-90, 50931 Cologne, Germany}
\email{chrisjenningsshaffer@gmail.com}

\author[K. Mahlburg]{Karl Mahlburg}
\address{Department of Mathematics, Louisiana State University, Baton Rouge, LA 70803, USA}
\email{mahlburg@math.lsu.edu}

\thispagestyle{empty} \vspace{.5cm}
\maketitle

\section{introduction and statement of results}

There is an extensive literature on the study of integer unimodal sequences: see \cite{Stan89} for a survey of combinatorial and other applications of such sequences, and \cite{BringmannMahlburg2} for a history of asymptotic results for the enumeration functions. More recently, there has been further interest in the asymptotic and probabilistic aspects of statistics for unimodal sequences. There is a much lengthier discussion in the authors' recent article \cite{BJSMR1}, which also includes results for strongly unimodal sequences. This article is intended to address the asymptotic behavior of statistics for several families of unimodal sequences that have not been studied previously, including, most importantly,
 the rank of
unimodal sequences with marked summits.

A sequence of positive integers is a {\it unimodal sequence} of size $n$ if it is of the form
\begin{equation}
\label{E:unimodDefn}
a_1 \leq a_2 \leq \cdots \leq a_r \leq \overline{c} \geq b_s \geq \cdots \geq b_1,
\end{equation}
with sum $a_1 + \cdots + a_r + c + b_s + \cdots + b_1 = n$. Let $\mathcal{U}(n)$ denote the set of unimodal sequences (with designated peaks) of size $n$, and let $u(n) := |\mathcal{U}(n)|$ be the enumeration function. This function has appeared previously as $ss(n)$ in \cite{BringmannMahlburg2}, as $\sigma\sigma(n)$ in Section 3 of \cite{And84}, as $v(n)$ in Section 2.5 of \cite{Sta11}, and as $X(n-1)$ in \cite{And12}. The mark on the part $c$ indicates that if the largest part is repeated, the sequences may be further distinguished by specifying the location of the ``peak''.  For example, $u(3)=6$, as the unimodal sequences are $\{\overline{3}\}$, $\{\overline{2},1\}$, $\{1,\overline{2}\}$, $\{\overline{1},1,1\}$,
$\{1,\overline{1},1\}$, and $\{1,1,\overline{1}\}$.
The generating function is given by
\begin{align*}
U(q) := \sum_{n\ge0} u(n)q^n
=
\sum_{n\ge0}\frac{q^n}{(q)_n^2}.
\end{align*}

Throughout the article we use the standard $q$-Pochhammer notation,
which is defined for $n\in\N_0\cup\{\infty\}$ by
\begin{gather*}
(a)_n := (a;q)_n := \prod_{j=0}^{n-1} \left(1-aq^j\right)
,\qquad\qquad
(a_1,a_2,\dotsc,a_k)_n := (a_1)_n(a_2)_n\dotsb(a_k)_n
.
\end{gather*}

The \emph{rank} of a unimodal sequence is the number of parts after the peak minus the number of
parts before the peak. As such, it is direct to see that the generating function is
\begin{align*}
U(\zeta;q)
&=\sum_{\substack{n\ge0\\ m\in\mathbb{Z}}} u(m,n) \zeta^m q^n
=\sum_{n\ge0} \frac{q^n}{(\zeta q,\zeta^{-1}q)_n},
\end{align*}
where $u(m,n)$\footnote{Unfortunately, there is no strongly established convention for the usage of $u(m,n)$ and $U(\zeta;q)$, which have been used to denote both unimodal sequences and {\it strongly unimodal sequences} (in which all inequalities in \eqref{E:unimodDefn} are strict).} denotes the number of unimodal sequences, with designated peaks, of
size $n$ and with rank $m$.
We denote the moments, absolute moments, and moment generating  functions, respectively, by
\begin{align*}
u_{k}(n) := \sum_{m\in\mathbb{Z}}m^{k} u(m,n)
,\qquad
u_{k}^+(n) := \sum_{m\in\mathbb{Z}}|m|^{k} u(m,n)
,\qquad
\mathbb{U}_{k}(q) := \sum_{n\ge0} u_{k}(n)q^n.
\end{align*}

The next family that we consider has a generating function whose analytic behavior is very similar to $U(\zeta;q)$, although this is not immediately evident from the combinatorial definition.
We define a \emph{Durfee unimodal sequence} to be a unimodal sequence with a designated peak as in \eqref{E:unimodDefn}, and
that also satisfies  $s \leq c-k$ where $k$ is the size of the ``Durfee square'' of the partition $(a_1,a_2,\dotsc,a_r)$ (this is the largest $k$ such $a_{r-k+1} \geq k$).  Let $\mathcal{V}(n)$ denote the set of Durfee unimodal sequences of size $n$, with enumeration function $v(n)$.  This function was introduced in \cite{KimLovejoy2}, where the enumeration function was denoted by $V(n)$.
Again the rank is the number of parts after the peak minus the number of parts before
the peak. Let $v(m,n)$ denote the number of Durfee unimodal sequences of size $n$ with rank $m$.
The generating function is given by (see \cite[Proposition 3.1]{KimLovejoy2})
\begin{gather*}
V(\zeta;q)
:= \sum_{\substack{n\ge0 \\ m\in\mathbb{Z}}} v(m,n)\zeta^m q^n
	=
	\sum_{n\ge0} \frac{(q^{n+1})_nq^n}{(\zeta q,\zeta^{-1} q)_n}
.
\end{gather*}
As above, we also introduce notation for the moments, namely
\begin{equation*}
v_k(n) := \sum_{m\in\mathbb{Z}} m^k v(m,n)
,\qquad
v_k^+(n) := \sum_{m\in\mathbb{Z}} |m|^k v(m,n)
,\qquad
\mathbb{V}_k(q) := \sum_{n\ge0} v_k(n)q^n.
\end{equation*}

The third and final family that we consider are \emph{semi-strict} unimodal sequences, which are of the form
\begin{gather*}
a_1 < a_2 < \dotsc a_r < c > b_s \ge b_{s-1} \ge \dotsc \ge b_1.
\end{gather*}
Let $\mathcal{DM}(n)$ denote the set of semi-strict unimodal sequences of size $n$, with enumeration function ${\rm dm}(n),$ as in \cite{BringmannMahlburg2}.  This function was introduced in \cite{And12}, where the enumeration function was written as $x_m(n)$.
For example, ${\rm dm}(4)=5$ from the sequences
$\{4\}$, $\{1,3\}$, $\{3,1\}$, $\{1,2,1\}$, and $\{2,1,1,1\}$.
The generating function is given by
\begin{gather*}
D(q) := \sum_{n\ge0} {\rm dm}(n)q^n = \sum_{n\ge0} \frac{(-q)_n}{(q)_n}q^{n+1}.
\end{gather*}
The rank of such a sequence is again defined as the number of parts after
the peak minus the number of parts before the peak. Let ${\rm dm}(m,n)$ denote the
number of semi-strict unimodal sequences of size $n$ with rank $m$.
The generating function
of this rank is given by

\begin{gather*}
D(\zeta;q) := \sum_{n\ge0} {\rm dm}(m,n)\zeta^m q^n = \sum_{n\ge0} \frac{\left(-\zeta^{-1} q\right)_n}{(\zeta q)_n}q^{n+1}.
\end{gather*}
Set
\begin{equation*}
{\rm dm}_k(n) := \sum_{m\in\mathbb{Z}} m^k {\rm dm}(m,n)
,\quad
{\rm dm}_k^+(n) := \sum_{m\in\mathbb{Z}} |m|^k {\rm dm}(m,n)
,\quad
\mathbb{D}_k(q) := \sum_{n\ge0} {\rm dm}_k(n)q^n.
\end{equation*}

Our first result gives the asymptotic behavior of the moments for the rank functions.
We note that the odd moments $u_{2k+1}(n)$ and $v_{2k+1}(n)$ are trivially all zero, since it is clear by symmetry that $u(-m,n)=u(m,n)$ and $v(-m,n)=v(m,n)$. We denote the $\ell$-th Bernoulli polynomial by $B_\ell(x)$.

\begin{theorem}\label{thm:AsymptoticForMoments}
Suppose that $k\in\N_0$. Then we have the following asymptotic formulas.
\begin{enumerate}[leftmargin=*, label={\rm (\arabic*)}]
	\item \label{thm:Asymp:u} We have, as $n\to \infty$,
	\begin{equation*}
	u_{2k}(n) \sim (-1)^k 2^{2k-3} 3^{k-\frac 34} B_{2k}\left(\frac12\right)
	n^{k-\frac 54} e^{2\pi \sqrt{\frac{n}{3}}}.
	\end{equation*}
	\item We have, as $n\to \infty$,
	\begin{equation*}
	v_{2k}(n) \sim (-1)^k 2^{2k-2} 3^{k-\frac74} B_{2k}\left(\frac12\right) n^{k-\frac 54} e^{2\pi \sqrt{\frac{n}{3}}}.
	\end{equation*}
	\item \label{thm:Asymp:semistrict}
We have, as $n\to \infty$,
	\begin{equation*}
	{\rm dm}_k(n) \sim \frac{1}{16 \pi^k} \log(n)^k n^{\frac{k}{2}-1}	e^{\pi \sqrt{n}}
	.
	\end{equation*}
\end{enumerate}
\end{theorem}

Following the probabilistic Method of Moments, we use the above asymptotic formulas to show that each of the ranks has a limiting distribution when appropriately normalized.

\begin{proposition}\label{prop:URankDistribution}
\begin{enumerate}[leftmargin=*, label={\rm (\arabic*)}]
\item
\label{prop:URankDistribution:unimodal}
The normalized rank of unimodal sequences is asymptotically distributed according to the logistic distribution with mean zero and scale $\frac{1}{\pi}$. In particular,
\begin{align*}
\lim_{n\rightarrow\infty} \frac{1}{u(n)}
\left|\left\{
	\sigma\in\mathcal{U}(n) : \frac{\rank(\sigma)}{\sqrt{3n}} \le x
\right\}\right|
=
\frac{1}{1+e^{-\pi x}}.
\end{align*}
\item
\label{prop:URankDistribution:Durfee}
The normalized rank of Durfee unimodal sequences is asymptotically distributed according to the logistic distribution with mean zero and scale $\frac{1}{\pi}$. In particular,
\begin{align*}
\lim_{n\rightarrow\infty} \frac{1}{v(n)}
\left|\left\{
	\sigma\in\mathcal{V}(n) : \frac{\rank(\sigma)}{\sqrt{3n}} \le x
\right\}\right|
=
\frac{1}{1+e^{-\pi x}}.
\end{align*}
\item
\label{prop:URankDistribution:semistrict}
The normalized rank of semi-strict unimodal sequences is asymptotically distributed according to a point mass distribution at one. In particular,
\begin{align*}
\lim_{n\rightarrow\infty} \frac{1}{{\rm dm}(n)}
\left|\left\{
	\sigma\in\mathcal{DM}(n) : \frac{\rank(\sigma)}{\frac{\sqrt{n} \log(n)}{\pi}} \le x
\right\}\right|
=
\begin{cases}
0 \quad &\text{if } x < 1, \\
1 \quad &\text{if } x \geq 1.
\end{cases}
\end{align*}
\end{enumerate}
\end{proposition}
\begin{remark}
The appearance of the logistic distribution in part \ref{prop:URankDistribution:unimodal} is not surprising, as this naturally arises as the difference between two independent extreme value distributions (see \cite{DJR13} for a related example). As in \cite[Proposition 2.5.1]{Sta11}, unimodal sequences are closely related to ordered pairs of partitions (up to an inclusion-exclusion argument), and the rank of the sequence then corresponds to the difference between the number of parts in the two partitions. Furthermore, if the partitions were independent, then \cite[Theorem 1.1]{EL41} would show that the normalized number of parts in each partition has a (weak) limit that is an extreme value distribution (see \cite[page 195]{Billingsley1}).
\end{remark}

\begin{remark}
The point mass distribution in part \ref{prop:URankDistribution:semistrict} can also be interpreted as the statement that for large $n$, ``almost all'' semi-strict unimodal sequences of $n$ have rank that is approximately $\frac{\sqrt{n} \log(n)}{\pi}.$
However, it would also be interesting to obtain the more refined distribution centered around this average value. In particular, following the example of \cite[Theorem 1.1]{EL41}, one might expect to find a limiting distribution $F(x)$ such that
\begin{equation*}
\lim_{n \to \infty} \frac{1}{{\rm dm}(n)} \left| \left\{\sigma \in \mathcal{DM}(n) \; : \; \frac{\rank(\sigma) - \frac{\sqrt{n} \log(n)}{\pi}}{\sqrt{n}} \leq x \right\} \right| \to F(x).
\end{equation*}
Such a result is not easily accessible using the techniques in this paper, as all of our calculations are instead for moments (and thus distributions) that are centered at zero.
\end{remark}

Our final result highlights an additional application of the Method of Moments, as we use the distributions from Proposition \ref{prop:URankDistribution} in order to determine asymptotic formulas for the absolute moments.
\begin{corollary}\label{cor:AbsoluteMoments}
Assume that $k\in \N_0$.
\begin{enumerate}[leftmargin=*, label={\rm (\arabic*)}]
	\item \label{cor:AbsoluteMoments:unimodal}
We have, as $n \to \infty,$
	\begin{equation*}
	u_k^+(n) \sim \frac{ 3^{\frac{k}{2}-\frac34} \left(1-2^{1-k}\right) k! \zeta(k)  }
	{4\pi^{k}} 	
	n^{\frac{k}{2}-\frac54}e^{2\pi \sqrt{\frac{n}{3}}}.
	\end{equation*}
	\item \label{cor:AbsoluteMoments:Durfee}
We have, as $n \to \infty,$
	\begin{equation*}
	v_k^+(n) \sim \frac{ 3^{\frac{k}{2}-\frac74} \left(1-2^{1-k}\right) k! \zeta(k)  }
	{2\pi^{k}} 	
	n^{\frac{k}{2}-\frac54}e^{2\pi \sqrt{\frac{n}{3}}}.
	\end{equation*}
	\item \label{cor:AbsoluteMoments:semistrict}
We have, as $n \to \infty,$
	\begin{equation*}
	{\rm dm}_k^+(n)\sim {\rm dm}_k(n).
	\end{equation*}
\end{enumerate}
\end{corollary}

The remainder of the paper is structured as follows.
In Section \ref{prefacts}, we recall some preliminary facts on the Dedekind $\eta$-function, Ingham's theorem,
the  Euler-Maclaurin summation formula, as well as some combinatorial statistics. In Section \ref{unimodal} we prove the asymptotic results for the rank of unimodal sequences.  This is followed by additional results on the log-concavity of the unimodal rank in  Section \ref{asymp}. Section \ref{durfee} is dedicated to proving our results for Durfee unimodal sequences, and we conclude with semi-strict unimodal sequences in Section \ref{semi}.

\section{Preliminary facts}{\label{prefacts}

\subsection{Asymptotic results.}

We  require the asymptotic behavior of {\it Dedekind's $\eta$-function}
$\eta(\tau):=q^\frac1{24}\prod\limits_{n=1}^\infty(1-q^n)$ $(q:=e^{2\pi i\tau},\tau\in\mathbb{H})$,
which follows directly from its modular transformation  \cite[Theorem 3.1]{Apo90}
\begin{equation}\label{Pochas}
\left(e^{-w};e^{-w}\right)_\infty \sim \sqrt{\frac{2\pi}{w}} e^{-\frac{\pi^2}{6w}} \qquad \mbox{as }w\to 0.
\end{equation}
Here the limit is taken in any region $|\!\Arg(w)|<\theta$, for fixed $\theta<\frac{\pi}{2}$. Throughout the article, almost all asymptotic statements for $w \to 0$ are based on setting $q=e^{-w}$.

Moreover we need the following Tauberian theorem, which is a special case of Ingham's more general result \mbox{\cite[Theorem $1'$]{Ingham1}}.
\begin{theorem}\label{Taub}
	Suppose that $B(q)=\sum_{n\ge0}b_nq^n$ is a power series with weakly increasing non-negative
	coefficients and radius of convergence at least one. If
	$\lambda$, $\alpha$, $\beta$, and $\gamma$ are real numbers with $\gamma>0$ such that
	\begin{gather*}
	B\left(e^{-t}\right) \sim \lambda \log\left( \frac{1}{t} \right)^\alpha t^\beta e^{\frac{\gamma}{t}}
	\quad\mbox{as } t\rightarrow0^+,
	\qquad\qquad
	B\left(e^{-z}\right) \ll \log\left( \frac{1}{|z|} \right)^\alpha |z|^\beta e^{\frac{\gamma}{|z|}}
	\quad\mbox{as } z\rightarrow0,
	\end{gather*}
	with $z=x+iy$ in each region of the form $|y|\le \Delta x$ with $\Delta>0$,
	then
	\begin{gather*}
	b_n
	\sim
	\frac{ \lambda   \gamma^{\frac{\beta}{2}+\frac14 }}
	{2^{\alpha+1}\sqrt{\pi}  }\log\left( n \right)^\alpha n^{-\frac{\beta}{2}-\frac34}
	e^{2\sqrt{\gamma n}}
	\qquad\qquad\mbox{as } n\rightarrow\infty.
	\end{gather*}

\end{theorem}
\begin{remark}
Theorem \ref{Taub} has been stated in a number of recent publications without the boundedness condition for  ``angular'' regions $|y| \leq \Delta x$, but this is in error, as the general statement does indeed require the additional bound. This was noted by Ingham \cite[p. 1088]{Ingham1}, and the technical aspects of the theorem are discussed in great detail in the authors' recent preprint \cite{BringmannJenningsShafferMahlburg1}. However, this point is of no great concern here, as Section 3.1 of \cite{BringmannJenningsShafferMahlburg1} also explains that if the asymptotic behavior of $f$ is determined by modular inversion, then the angular boundedness condition holds automatically.
\end{remark}

We next recall a result for asymptotic expansions that follows from the Euler-Maclaurin summation formula. Although this technique is widely used (see Section 6.7.4 of \cite{Zagier1}), the only source we aware of that provides a proof for all of the following results is the authors' recent preprint \cite{BringmannJenningsShafferMahlburg1}.
We say that a function $f$ is of {\it sufficient decay} in a domain $D\subset\C$ if there exists some $\varepsilon > 0$
such that $f(w) \ll w^{-1-\varepsilon}$ as $|w| \to \infty$ in $D$.
\begin{proposition}\label{Prop:EulerMaclaurin1DShifted}
	Suppose that  $0\le \theta < \frac{\pi}{2}$ and let
	$D_\theta := \{ re^{i\alpha} : r\ge0 \mbox{ and } |\alpha|\le \theta  \}$.
	Let $f:\C\rightarrow\C$ be holomorphic in a domain containing
	$D_\theta$, so in particular $f$ is holomorphic at the origin, and
	assume that $w \mapsto f(w)$ and all of its derivatives are of sufficient decay.
	Then for $a\in\mathbb{R}$ and $N\in\N_0$,
	\begin{gather*}
	\sum_{m\geq0}f(w(m+a))
	=
	\frac{1}{w}\int_0^\infty f(x) dx
	-
	\sum_{n=0}^{N-1} \frac{B_{n+1}(a) f^{(n)}(0)}{(n+1)!}w^n
	+
	O_N\left(w^N\right)
	,
	\end{gather*}
	uniformly, as $w\rightarrow0$ in $D_\theta$.
\end{proposition}

 A useful corollary also gives a compact expression in the case of alternating signs.
\begin{corollary}\label{Cor:EulerMaclaurinBooleShift}
Under the assumptions and notation of Proposition \ref{Prop:EulerMaclaurin1DShifted}, we have
	\begin{gather*}
	\sum_{m\geq0} (-1)^m f(w(m+a))
	=
	\frac{1}{2}\sum_{n=0}^{N-1} \frac{E_{n}(a) f^{(n)}(0)}{n!}w^n
	+
	O_N\left(w^N\right)
	,
	\end{gather*}
	uniformly, as $w\rightarrow0$ in $D_\theta$, where
	$E_n(x)$ are the Euler polynomials.
\end{corollary}

We  also need the asymptotic expansion in the case that the function has a simple pole at zero. For $a \not \in -\N_0$, define the constant
$C_a :=-\gamma-\psi(a)$, where $\psi(a):= \frac{\Gamma'(a)}{\Gamma(a)}$
is the {\it digamma function} \cite[equation 6.3.16]{AS64}, and $\gamma$ is the Euler-Mascheroni constant.
\begin{proposition}
	\label{Prop:EulerMaclaurin1DPoleShifted}
	Suppose that $0\le \theta < \frac{\pi}{2}$.
	Let $f:\C\rightarrow\C$ be holomorphic in a domain containing
	$D_\theta$, except for a simple pole at the origin,
	and assume that $w \mapsto f(w)$ and all of its derivatives are of sufficient decay
	as $|w|\rightarrow\infty$ in $D_\theta$.
	If
	$
	f(w) = \sum_{n\geq-1} b_{n}w^n,
$
	then for $a\in\mathbb{R}$, with $a\not\in-\N_0$,
	and $N\in\N_0$, uniformly, as $w\rightarrow0$ in $D_\theta$,
	\begin{multline*}
	\sum_{m\geq0} f(w(m+a))
	=-
	\frac{b_{-1}\Log(w)}{w}
	+
	\frac{b_{-1}C_a}{w}
	+
	\frac{1}{w}\int_0^\infty \left( f(x) - \frac{b_{-1}e^{-x}}{x}\right) dx
	\\	
	-
	\sum_{n=0}^{N-1} \frac{B_{n+1}(a) b_n}{n+1}w^n+
	O_N\left(w^N\right)
	.
	\end{multline*}
\end{proposition}

\subsection{Partition statistics}

We  use several basic definitions from the theory of integer partitions. For a partition $\lambda = \lambda_1 + \dots + \lambda_\ell$, with parts written in weakly decreasing order, Dyson \cite{Dyson1} defined its {\it rank} to be
\begin{equation*}
\rank(\lambda) := \lambda_1 - \ell.
\end{equation*}
Let $N(m,n)$ denote the number of partitions of $n$ with rank $m$.

Furthermore, let $\omega(\lambda)$ denote the number of ones in $\lambda$, and let $\mu(\lambda)$ be the number of parts larger than $\omega(\lambda)$. As defined by Andrews and Garvan \cite{AndrewsGarvan1} (and building on Garvan's earlier work on the ``vector crank'' \cite{Gar}), the {\it crank} of the partition is then given by
\begin{equation*}
\crank(\lambda) := \begin{cases} \lambda_1 \quad &\text{if } \omega(\lambda) = 0, \\
\mu(\lambda) - \omega(\lambda) &\text{if } \omega(\lambda) \ge 1.
\end{cases}
\end{equation*}
Let $M(m,n)$ denote the number of partitions of $n$ with crank $m$ (with a slight modification in the case that $n = 1$, where the values are instead  $M(\pm 1, 1) = 1, M(0,1) = -1$).

As was thoroughly discussed in \cite{BJSMR1}, the partition rank is combinatorially related to the unimodal rank, and it therefore is somewhat surprising that it is instead the generating function of the crank that plays a major role in the asymptotic analysis of the unimodal rank. The following product formula is found in \cite{AndrewsGarvan1}:
\begin{align}
\label{E:crank}
C(\zeta;q)
:= \sum_{\substack{ n \geq 0\\ m \in \Z}} M(m,n) \zeta^m q^n
= \frac{(q)_\infty}{(\zeta q, \zeta^{-1} q)_\infty}.
\end{align}

\section{Unimodal sequences}\label{unimodal}
\label{Sec:unimodal}

In this section we prove the asymptotic results for the rank of unimodal sequences, beginning with the moments. Throughout we write $\zeta=e^{z}$, so that
\begin{gather*}
\left[ \partial_z^k \left(\zeta^m\right) \right]_{z=0} = m^k,\qquad \textnormal{where }\partial_z:=\frac{\partial}{\partial z}.
\end{gather*}

\begin{proof}[Proof of Theorem {\rm \ref{thm:AsymptoticForMoments} (1)}]
Since $u_{2k}(n)\le u_{2k}(n+1)$ (as can be seen by adding one to the peak), we can determine
the asymptotic value of $u_{2k}(n)$ by Theorem \ref{Taub},
once we know the
asymptotic main term of $\mathbb{U}_{2k}(e^{-w})$ as $w\rightarrow0$
with $|\Arg(w)|\le \theta<\frac{\pi}{2}$.

For this, we state the following identity from Proposition 2.1 of \cite{KimLovejoy1}
\begin{align}
\label{E:Ugen}
U(\zeta;q)
&= \frac{(q)_\infty}{\left(\zeta q,\zeta^{-1}q\right)_\infty}G_1(\zeta;q) + H_1(\zeta;q),
\end{align}
where
\begin{align*}
G_1(\zeta;q)
&:=
	\frac{1}{(q)_\infty}
	\sum_{n\ge0} (-1)^n \zeta^{2n+1} q^{\frac{n(n+1)}{2}},
\\
H_1(\zeta;q)
&:=
	(1-\zeta)\sum_{n\ge0} (-1)^n \zeta^{3n} q^{\frac{n(3n+1)}{2}} \left(1-\zeta^2q^{2n+1}\right)
.
\end{align*}
Note that the product in \eqref{E:Ugen} is exactly $C(\zeta; q)$, as in \eqref{E:crank}.

Since $\mathbb{U}_{2k}(q) = [ \partial_z^{2k} (U(\zeta;q))]_{z=0}$, we compute, for $\ell \in \N_0$
\begin{align*}
\left[\partial_z^\ell (G_1(\zeta;q)) \right]_{z=0}
&=
	\frac{1}{(q)_\infty}
	\sum_{n\ge0} (-1)^n (2n+1)^{\ell} q^{\frac{n(n+1)}{2}}
,\\
\left[ \partial_z^\ell(H_1(\zeta;q)) \right]_{z=0}
&=
	\sum_{n\ge0} (-1)^n q^{\frac{n(3n+1)}{2}}
	\left( (3n)^{\ell} - (3n+1)^{\ell} - (3n+2)^{\ell}q^{2n+1} + (3n+3)^{\ell}q^{2n+1}	
	\right)
.
\end{align*}	
Denoting the $\ell$-th moment generating function of the crank by
\begin{equation*}
C_{\ell}(q) := \sum_{\substack{n \geq 0 \\ m \in \Z}} m^\ell M(m,n) q^n,
\end{equation*}
we obtain, using the product rule 
\begin{align}\label{Eq:UMomentsReadyForAsymptotics}
\mathbb{U}_{2k}(q)
&=
	\left[ \partial_z^{2k} \left( C(\zeta;q)G_1(\zeta;q)+H_1(\zeta;q) \right)\right]_{z=0}
\\
&=
	\sum_{j=0}^{2k} \binom{2k}{j} C_{j}(q)
	\left[ \partial_z^{2k-j} (G_1(\zeta;q)) \right]_{z=0}
	\nonumber\\&\quad
	+
	\sum_{n\ge0} (-1)^n q^{\frac{n(3n+1)}{2}}
	\left( (3n)^{2k} - (3n+1)^{2k} - (3n+2)^{2k}q^{2n+1} + (3n+3)^{2k}q^{2n+1}	
	\right)
\nonumber\\
&=
	\sum_{j=0}^{k} \binom{2k}{2j} C_{2j}(q)
	\left[ \partial_{z}^{2(k-j)} (G_1(\zeta;q)) \right]_{z=0}
	\nonumber\\&\quad
	+
	\sum_{n\ge0} (-1)^n q^{\frac{n(3n+1)}{2}}
	\left( (3n)^{2k} - (3n+1)^{2k} - (3n+2)^{2k}q^{2n+1} + (3n+3)^{2k}q^{2n+1}	
	\right)
\nonumber\\
&=
	\frac{1}{(q)_\infty}
	\sum_{j=0}^{k} \binom{2k}{2j} C_{2j}(q)
	\sum_{n\ge0} (-1)^n (2n+1)^{2(k-j)} q^{\frac{n(n+1)}{2}}
	\nonumber\\&\quad
	+
\sum_{n\ge0} (-1)^n q^{\frac{n(3n+1)}{2}}	
	\left( (3n)^{2k} - (3n+1)^{2k} - (3n+2)^{2k}q^{2n+1}  \nonumber + (3n+3)^{2k}q^{2n+1}	
	\right)
,
\end{align}
where in the penultimate equality we use the fact that the odd moments of the crank generating function vanish.
	
We now determine the asymptotic behavior of the individual components. Using Proposition \ref{Prop:EulerMaclaurin1DShifted}, we may show that the second sum is bounded by the polynomial order $O(w^{-\frac{\ell}{2}})$ for some $\ell \in  \Z$, as $w \to 0$.
For the first term we first determine the behavior of the crank moments. Corollary 3.3 of \cite{BMR} implies that
\begin{equation}
\label{E:C2jAsymp}
C_{2j}\left( e^{-w} \right)
\sim
	(-1)^j B_{2j} \left( \frac 12 \right) \left( \frac{\vphantom{1}w}{2\pi}\right)^{\frac 12 -2j} e^{\frac{\pi^2}{6w}}
,
\end{equation}
where the limit can be taken in any region with $|\Arg(w)|\le \theta<\frac{\pi}{2}$.

Next we determine the behavior of
\begin{align*}
F_j(w)
&:=
	w^{k-j} 2^{2(j-k)} e^{-\frac{w}{8}}
	\sum_{n\geq 0} (-1)^n (2n+1)^{2(k-j)} e^{-\frac{n(n+1)w}{2}}=
	\sum_{n\geq 0} (-1)^n f_{k-j}\left(\sqrt{w}\left(n+\frac12\right) \right)
,
\end{align*}
where
\begin{equation*}
f_\ell(w) :=w^{2\ell}e^{-\frac{w^2}{2}}.
\end{equation*}
From Corollary \ref{Cor:EulerMaclaurinBooleShift}
we obtain  (because $E_{2n+1}(\frac12) = 0$) that
\begin{align}\label{asF}
F_j(w)
\sim
	\frac{1}{2}E_{2k-2j}\left(\frac12\right) w^{k-j}
.
\end{align}
Using \eqref{Pochas}, \eqref{E:C2jAsymp}, and \eqref{asF} gives that the first summand
of \eqref{Eq:UMomentsReadyForAsymptotics} (with $q=e^{-w}$) equals
\begin{multline*}
\frac{1}{\left( e^{-w};e^{-w}\right)_\infty}
	\sum_{j=0}^{k} \binom{2k}{2j} C_{2j} \left(e^{-w} \right) w^{j-k} 2^{2(k-j)} e^{\frac{w}{8}} F_j(w)
\\
\sim
	e^{\frac{\pi^2}{3w}} \sum_{j=0}^k \binom{2k}{2j} (-1)^{j} 2^{2k-2j-1}
	B_{2j}\left(\frac12\right) E_{2k-2j}\left(\frac12\right)
	\left(\frac{\vphantom{1}w}{2\pi}\right)^{1-2j}.
\end{multline*}
The $j=k$ term is dominant giving
\begin{gather*}
\frac12 (-1)^k B_{2k}\left(\frac12\right) \left(\frac{\vphantom{1}w}{2\pi}\right)^{1-2k} e^{\frac{\pi^2}{3w}}.
\end{gather*}

Applying Theorem \ref{Taub} then gives
the claimed asymptotic formula.
\end{proof}

We next turn to the proof of the limiting distribution for unimodal sequences. As in Section 4 of \cite{BJSMR1}, we use the probabilistic Method of Moments, which essentially employs  the limiting behavior of the moments of a sequence of random variables in order to
determine the limiting distribution (see Section 30 of \cite{Billingsley1}).
\begin{proof}[Proof of Proposition \ref{prop:URankDistribution} \ref{prop:URankDistribution:unimodal}]
The asymptotic formula for unimodal sequences with marked peaks is given by the case $k=0$ in Theorem \ref{thm:AsymptoticForMoments} \ref{thm:Asymp:u}, namely
\begin{align*}
u(n)
\sim 8^{-1} 3^{-\frac34} n^{-\frac54} e^{2\pi\sqrt{\frac{n}{3}}}.
\end{align*}
To the best of our knowledge, this expression first appeared in print as (5.1) in \cite{BringmannMahlburg2}.\footnote{However, as was further explained in Section 5 of \cite{BringmannMahlburg2}, the formula directly follows from earlier work of Stanley \cite{Sta11} and Wright \cite{Wri71}.}

Combining with the case of general $k$ in Theorem \ref{thm:AsymptoticForMoments}  \ref{thm:Asymp:u}, and using the relation $B_k(\frac12)=(2^{1-k}-1)B_k$ (see e.g. \cite[23.1.21]{AS64}), we therefore have, as $n\to\infty$
\begin{align*}
\frac{u_{2k}(n)}{u(n)} \sim  (3n)^k \left(2^{2k}-2\right)(-1)^{k+1}B_{2k}
.
\end{align*}
Since $(-1)^{k+1}B_{2k}>0$ \cite[23.1.15]{AS64}, we conclude that
\begin{align*}
\frac{u_{k}(n)}{(3n)^{\frac{k}{2}} u(n)}
\sim
	\left(2^k-2\right)|B_{k}|,
\end{align*}
as for $k$ odd  this is trivially true.
However, $(2^k-2)|B_{k}|$ is well-known to be the $k$-th
moment for the logistic distribution, with mean $\mu=0$ and scale $s=\frac{1}{\pi}$ (see \cite[p. 116--118]{JKB}),
and thus the the proof is complete upon applying the Method of Moments.
\end{proof}

Finally, we use the limiting distribution from above in order to calculate the asymptotic behavior of the absolute moments for the rank of unimodal sequences.

\begin{proof}[Proof of Corollary \ref{cor:AbsoluteMoments} {\rm(1)}]
Let $X_n$ denote the random variable defined by
\begin{gather*}
X_n(\sigma) := \frac{\rank(\sigma)}{\sqrt{3n}},
\end{gather*}
for $\sigma\in\mathcal{U}(n)$, with each $\sigma$ occuring with the uniform
probability $\frac{1}{u(n)}$,  and $X$ denote the random variable associated
to the logistic distribution.

The Method of Moments implies that $X_n$ converges in distribution to $X$. By the Continuous Mapping Theorem,
$|X_n|$ converges in distribution to $|X|$. By the corollary to Theorem 25.12
of \cite{Billingsley1}, if $\sup_{n\in\N}E[|X_n|^{r+\varepsilon}]<\infty$
for some $\varepsilon>0$, then $E[|X_n|^r]\rightarrow E[|X|^r]$.

For fixed $r$ we take $\varepsilon=1$ if $r$ is odd and $\varepsilon=2$
if $r$ is even. By doing so we have
\begin{align*}
\sup_{n\in\N} E\left[|X_n|^{r+\varepsilon}\right]
=
\sup_{n\in\N} E\left[X_n^{r+\varepsilon}\right]
,
\end{align*}
which is finite since
\begin{gather*}
\lim_{n\rightarrow\infty}
E\left[X_n^{r+\varepsilon}\right]
=
	\lim_{n\rightarrow\infty}
	\frac{u_{r+\varepsilon}(n)}{(3n)^{\frac{r+\varepsilon}{2}} u(n)}
=
	\left(2^{r+\varepsilon}-2\right)|B_{r+\varepsilon}|
.
\end{gather*}
Thus, with $\zeta$ denoting the Riemann zeta function,
\begin{gather*}
\lim_{n\rightarrow\infty} \frac{u^+_k(n)}{(3n)^{\frac{k}{2}} u(n)}
=
	E\left[\left\lvert X \right\rvert^k\right]
=
	2\Gamma(k+1) \pi^{-k} \left(1-2^{1-k}\right) \zeta(k)
,
\end{gather*}
where the formula for the absolute moments of the logistic distribution
was given in \cite[equation (23.11)]{JKB}.
\end{proof}

\section{Asymptotics for $u(m,n)$}\label{asymp}

In \cite{BJSMR1}, the authors conjectured the strict log-concavity of the rank
of strongly unimodal sequences and verified the conjecture in a limiting sense.
The same phenomenon appears to occur for $u(m,n)$.
\begin{conjecture}\label{ConjectureLogConcave}
For $n\ge37$ and $|m|\le n-23$, we have
\begin{gather*}
u(m,n)^2 > u(m-1,n)u(m+1,n).
\end{gather*}
\end{conjecture}

To see that the conjecture is reasonable, we now show that it holds for $n$ sufficiently large (depending on $m$).
\begin{lemma}\label{lem:LogConcave}
	For $n\to\infty$, Conjecture \ref{ConjectureLogConcave} is true.
\end{lemma}

\begin{proof}
Recall that Corollary 6.4 of \cite{BringmannKim} states that
\begin{align*}
u(m,n)=\frac{\pi^2}{2} X_3(n)+ \frac{\pi^3}{3} X_4(n)+ \frac{\pi^4}{72} \left( 59-36m^2 \right) X_5(n) + O_m \left( n^{-3} e^{2\pi \sqrt{\frac n3}}\right),
\end{align*}
where, with $I_\kappa(x)$ the $I$-Bessel function of order $\kappa$,
\begin{align*}
X_j(n):=\left( 2 \sqrt{3n}\right)^{-j} I_{-j}\left( 2\pi \sqrt{\frac{n}{3}}\right).
\end{align*}
Using that $I_{-j}(x)= (2\pi x)^{-\frac12}e^{x}(1+O(x^{-1}))$ as $x\to\infty$, we obtain that
\begin{gather*}
u(m,n)^2 - u(m-1,n)u(m+1,n)
=
	\frac{\pi^6}{2}X_3(n)X_5(n) + O_m\left( n^{-\frac{19}{4}}e^{4\pi \sqrt{\frac{n}{3}}}\right).
\end{gather*}
This gives the claim.
\end{proof}

We note that similar statements appear to hold for both the rank and crank of ordinary partitions, which we record here for posterity.
\begin{conjecture}
\label{C:NMLogConcave}
The following inequalities hold:
\begin{align*}
N(m,n)^2 &> N(m-1,n)N(m+1,n)& \qquad\qquad&\mbox{for $n\ge123$ and $|m|\le n-72$}
,\\
M(m,n)^2 &> M(m-1,n)M(m+1,n)& \qquad\qquad&\mbox{for $n\ge125$ and $|m|\le n-71$}.
\end{align*}
\end{conjecture}
There are related results for the partition function $p(n)$, which is known to be log-concave for $n > 25$; this was originally proven by Nicolas \cite{Nic78} (also see \cite{DesalvoPak1}). The proof relies on certain analytic properties of the asymptotic growth of $p(n)$, and the recent work of Griffin-Ono-Rolen-Zagier \cite{GORZ} on the hyperbolicity of polynomials associated to real sequences shows that a more general phenomenon holds for a wide class of sequences. Unfortunately, these analytic techniques do not seem to directly apply to the statistics in Conjectures \ref{ConjectureLogConcave} and \ref{C:NMLogConcave}.

\section{Durfee Unimodal Sequences}\label{durfee}

In this section we consider Durfee unimodal sequences, which turn out to have many similarities to unrestricted unimodal sequences. We begin by proving the asymptotic formulas for the moments of the Durfee unimodal rank.

\begin{proof}[Proof of Theorem \ref{thm:AsymptoticForMoments} {\rm (2)}]

As with $u_{2k}(m)$, we see that $v_{2k}(n)\le v_{2k}(n+1)$, by adding one to the peak.
We therefore again look to apply Theorem \ref{Taub},
by determining the asymptotic main term of $\mathbb{V}_{2k}(e^{-w})$ as $w\rightarrow0$
with $|\Arg(w)|\le \theta<\frac{\pi}{2}$.

For this, we use the following identity from Proposition 3.1 of \cite{KimLovejoy2}, for
$V(\zeta;q)$:
\begin{align*}
V(\zeta;q)
&=
C\left(\zeta;q\right) G_2(\zeta;q) + H_2(\zeta;q),
\end{align*}
where
\begin{align*}
G_2(\zeta;q)
:=
	\frac{1}{(q)_\infty}
	\sum_{n\ge0} \zeta^{3n+1} q^{3n^2+2n} \left(1-\zeta q^{2n+1}\right)
,
\qquad\qquad
H_2(\zeta;q)
:=
	(1-\zeta)\sum_{n\ge0} \zeta^{n} q^{n^2+n}
.
\end{align*}

Since $\mathbb{V}_{2k}(q) = [ \partial_z^{2k}( V(\zeta;q))]_{z=0}$,
we compute
\begin{align*}
\left[ \partial_z^\ell (G_2(\zeta;q)) \right]_{z=0}
&=
	\frac{1}{(q)_\infty}
	\sum_{n\ge0} \left( (3n+1)^{\ell} - (3n+2)^\ell q^{2n+1} \right) q^{3n^2+2n}
,\\
\left[ \partial_z^\ell (H_2(\zeta;q)) \right]_{z=0}
&=
	\sum_{n\ge0}
	\left( n^{\ell} - (n+1)^{\ell} \right) q^{n^2+n}	
.
\end{align*}	
Thus
\begin{align}\label{Eq:VMomentsReadyForAsymptotics}
\mathbb{V}_{2k}(q)
&=
	\left[ \partial_z^{2k}\left( C(\zeta;q)G_2(\zeta;q)+H_2(\zeta;q) \right)\right]_{z=0}
\nonumber\\
&=
	\frac{1}{(q)_\infty}
	\sum_{j=0}^{k} \binom{2k}{2j} C_{2j}(q)
	\sum_{n\ge0} \left( (3n+1)^{2(k-j)} - (3n+2)^{2(k-j)}q^{2n+1} \right)q^{3n^2+2n}
	\nonumber\\&\quad
	+
	\sum_{n\ge0} \left( n^{2k} - (n+1)^{2k}\right) q^{n^2+n}
,
\end{align}
where we again use that the odd moments of the crank generating function
are zero.
	
We now determine the asymptotics of the individual components. Proposition \ref{Prop:EulerMaclaurin1DShifted} implies that the second term is $O(w^{- \frac \ell2})$ for some $\ell \in \Z$ as $w \to 0$. For the first term we first determine, using Proposition \ref{Prop:EulerMaclaurin1DShifted} the asymptotic behavior of
\begin{align*}
\sum_{n\ge0} \left(f_\ell\left(\sqrt{w}\left( n+\frac13  \right)\right) - f_\ell\left(\sqrt{w}\left( n+\frac23  \right)\right)\right)\sim
-\frac{2 B_{2\ell+1}\left(\frac13\right)}{2\ell+1} w^{\ell}
,
\end{align*}
where $f_\ell(w) := w^{2\ell} e^{-3w^2}$.
Combining this with  \eqref{Pochas} and \eqref{E:C2jAsymp}
gives that the first summand
in \eqref{Eq:VMomentsReadyForAsymptotics}
is asymptotically  equal to
\begin{align*}
	-2 e^{\frac{\pi^2}{3w}}
	\sum_{j=0}^{k} \binom{2k}{2j} (-1)^j B_{2j}\left(\frac 12 \right)
	3^{2(k-j)} \frac{B_{2(k-j)+1}\left( \frac 13\right) } {2(k-j)+1} \left( \frac{w}{2\pi}\right)^{1 -2j}
.
\end{align*}
The $j=k$ term is dominant giving
\begin{gather*}
\frac13 (-1)^k B_{2k}\left(\frac12\right) \left(\frac{w}{2\pi}\right)^{1-2k} e^{\frac{\pi^2}{3w}}.
\end{gather*}

Applying Theorem \ref{Taub} we then obtain the claim.
\end{proof}

We conclude our discussion of Durfee unimodal sequences by noting that the proofs of Proposition \ref{prop:URankDistribution} \ref{prop:URankDistribution:Durfee} and Corollary \ref{cor:AbsoluteMoments} \ref{cor:AbsoluteMoments:Durfee} are essentially identical to the corresponding proofs for unimodal sequences from Section \ref{Sec:unimodal}.

\section{Semi-Strict Unimodal Sequences}\label{semi}

In this section, we investigate semi-strict unimodal sequences. We begin by proving Theorem \ref{thm:AsymptoticForMoments} {\rm (3)}.

\begin{proof}[Proof of Theorem \ref{thm:AsymptoticForMoments} {\rm (3)}]

Since the corresponding rank is monotone in $n$, i.e., ${\rm dm}(m,n)\le {\rm dm}(m,n+1)$ (again by adding one to the peak),
we can again apply Theorem \ref{Taub} to determine the asymptotics of the moments.

For this, we need to determine
\begin{equation*}
\lim_{w\to 0} \mathbb{D}_k\left(e^{-w}\right).
\end{equation*}
 Letting $x=q$, $\beta=-\zeta^{-1}q$, and $\gamma=\zeta q$ in
\cite[equation (4.1)]{AndrewsSubbaraoVidyasagar1} yields\begin{align*}
D(\zeta;q)
&=
D^*(\zeta;q)
+
\frac{q(1-\zeta^{-1})}{1+\zeta^{-2}q}
,
\end{align*}
where
\begin{equation*}
D^*(\zeta;q) := \frac{q\left(-\zeta^{-1}q\right)_\infty}{\zeta\left(1+\zeta^{-2}q\right)(\zeta q)_\infty}.
\end{equation*}
For $k=0$, we obtain, using \eqref{Pochas},
\begin{align*}
\label{E:D0Asymp}
\mathbb{D}_0\left(e^{-w}\right)=\frac14 \sqrt{\frac{w}{\pi}} e^{\frac{\pi^2}{4w}}.
\end{align*}

We next turn to higher $k$.
We compute the logarithmic derivative of $D^*$ as
\begin{equation}\label{L}
\frac{\partial_z\left(D^*(\zeta;q)\right)}{D^*(\zeta;q)}= -1 + \frac{2\zeta^{-2}q}{1+\zeta^{-2}q} - \sum_{n\geq 1} \frac{\zeta^{-1} q^n}{1+\zeta^{-1}q^n}+ \sum_{n\geq 1} \frac{\zeta q^n}{1-\zeta q^n} =: L(\zeta;q).
\end{equation}

We first consider the third term and set
\begin{equation*}
L_1(\zeta;q) := \sum_{n\geq1} \frac{\zeta^{-1} q^n}{1+\zeta^{-1} q^n} = -\sum_{n_1,n_2\geq 1}\left(-\zeta^{-1}q^{n_1}\right)^{n_2},
\end{equation*}
which is valid for $|q| < |\zeta|$ (in fact, we set $\zeta$ to be $1$ below).
In order to calculate the moments, we need the following derivatives for $\ell\in\N_0$:
\begin{align*}
L_{1,\ell}(q)
:=
	\left[ \partial_z^\ell\left( L_1(\zeta;q)\right)\right]_{z=0}
=
	(-1)^{\ell+1} \sum_{n_1,n_2\geq 1} n_2^\ell (-1)^{n_2} q^{n_1n_2}
=
	(-1)^{\ell+1} \sum_{n_2\geq1} \frac{n_2^{\ell}(-1)^{n_2}q^{n_2}}{1-q^{n_2}}.
\end{align*}

We next determine the asymptotic behavior of $L_{1,\ell}(e^{-w})$ as $w\to 0$ with
$|\Arg(w)|\le \theta<\frac{\pi}{2}$. The case $\ell=0$ is combined below with the fourth term. For $\ell\geq 1$ we write
\begin{align*}
L_{1,\ell}\left(e^{-w}\right)
=
	(-1)^{\ell}  w^{-\ell} \sum_{n\geq 0} (-1)^n f_\ell(w(n+1))
,
\end{align*}
where
\begin{equation*}
f_\ell(w) := \frac{w^\ell e^{-w}}{1-e^{-w}}.
\end{equation*}
Since $\ell\geq 1$, $f_\ell(w)$ does not have a pole at $w=0$ and we may apply Corollary \ref{Cor:EulerMaclaurinBooleShift} to obtain
\begin{equation*}
\sum_{n\geq 0} (-1)^n f_\ell(w(n+1))
= \frac12 \sum_{n=0}^{N-1} \frac{E_{n}(1) f_\ell^{(n)}(0)}{n!}  w^n + O\left(w^N\right).
\end{equation*}
Now $f^{(n)}(0)=0$ for $n<\ell-1$.
Thus
\begin{gather*}
\sum_{n\geq 0} (-1)^n f_\ell(w(n+1))
\ll
	w^{\ell-1},
\end{gather*}
which implies that
\begin{align*}
L_{1,\ell}\left(e^{-w}\right)
\ll
\frac{1}{w}.
\end{align*}

We next consider the fourth term in \eqref{L}, which we denote by
\begin{equation*}
L_2(\zeta;q) := \sum_{n\geq 1} \frac{\zeta q^n}{1-\zeta q^n}.
\end{equation*}
Proceeding as for $L_1$ we have for $\ell\in\N_0$
\begin{align*}
L_{2,\ell}(q)
:=
	\left[\partial_z^\ell\left(L_2(\zeta;q)\right)\right]_{z=0}
=
	\sum_{n\geq 1} \frac{n^\ell q^{n}}{1-q^{n}}.
\end{align*}
For $\ell\ge1$, we have by \cite[equation (6.80)]{Zagier1}, after correcting minor typos, that
\begin{gather}\label{AsL2}
L_{2,\ell}(e^{-w})
\sim
	\frac{\ell! \zeta(\ell+1)}{w^{\ell+1}}
.
\end{gather}

We now consider the $\ell=0$ cases of $L_{1;\ell}$ and $L_{2,\ell}$. For this, we note that
\begin{gather*}
L(1;q)
=
	-1 + \frac{2q}{1+q} - \sum_{n\ge1} \frac{q^n}{1+q^n} + \sum_{n\ge1} \frac{q^n}{1-q^n}
=
	-1 + \frac{2q}{1+q} + 2\sum_{n\ge1} \frac{q^{2n}}{1-q^{2n}}.
\end{gather*}
To determine the asymptotic behavior of $L(1;q)$, we write
\begin{gather*}
\sum_{n\ge1} \frac{q^{2n}}{1-q^{2n}}
=
\sum_{n\ge0} f(w(n+1)),
\end{gather*}
where
\begin{gather*}
f(w)
:=
	\frac{e^{-2w}}{1-e^{-2w}}
=
	\frac{1}{2w} - \frac12 + O(w).
\end{gather*}
By Proposition \ref{Prop:EulerMaclaurin1DPoleShifted},
we have
\begin{gather}\label{Asel}
L(1;e^{-w})
\sim
-\frac{\Log(w)}{w}.
\end{gather}

Recalling \eqref{L}, we have that for $k \in \N$,
\begin{align*}
\partial_z^{k}\left( D^*(\zeta;q)\right)
&=
	\sum_{j=0}^{k-1} \binom{k-1}{j} \partial_z^j\left( D^*(\zeta;q)\right) \partial_z^{k-1-j}\left(L(\zeta;q)\right).
\end{align*}
By induction, this implies that
\begin{gather*}
\partial_z^k\left( D^*(\zeta;q)\right)
=
	D^*(\zeta;q)
	\sum_{\ell_1 + 2\ell_2 + \dotsb + k\ell_k=k}	
	a(\ell_1, \ell_2, \dotsc, \ell_k) \prod_{n=0}^{k-1} \left( \partial_z^n\left( L(\zeta;q) \right) \right)^{\ell_{n+1}}	
,
\end{gather*}
where $\ell_h\in\N_0$ and the
$a(\ell_1, \ell_2, \dotsc, \ell_k)$ are constants.
Now for a given sequence of non-negative integers with
$\ell_1 + 2\ell_2 + \dotsb + k\ell_k=k$,
we have, using \eqref{Asel} and \eqref{AsL2}
\begin{align*}
\left[\prod_{n=0}^{k-1} \left(\partial_z^n\left( L(\zeta;e^{-w})\right)  \right)^{\ell_{n+1}}\right]_{z=0}
&\sim
	\left( \frac{\Log\left(\frac{1}{w}\right)}{w}  \right)^{\ell_1}
	\prod_{n=1}^{k-1} \left( \frac{n!\zeta(n+1)}{w^{n+1}}  \right)^{\ell_{n+1}}
\\
&=
 	\Log\left(\frac{1}{w}\right)^{\ell_1} w^{-k}
	\prod_{n=1}^{k-1} \left( n!\zeta(n+1)  \right)^{\ell_{n+1}}
.
\end{align*}
Since this is largest for $\ell_1=k$, and clearly $a(k,0,\dotsc,0)=1$, we have, using \eqref{Pochas}
\begin{align*}
\mathbb{D}_k(e^{-w})
&\sim
	\left[ \partial_z^k\left( D^*(\zeta;e^{-w})\right) \right]_{z=0}
\sim
	D^*(1;e^{-w}) \Log\left(\frac{1}{w}\right)^k w^{-k}
\sim
	\frac{1}{4\sqrt{\pi}} \Log\left(\frac{1}{w}\right)^k w^{\frac12-k}
	e^{\frac{\pi ^2}{4w}}
.
\end{align*}

Applying Theorem \ref{Taub} then yields  the claim.
\end{proof}

\begin{proof}[Proof of Proposition \ref{prop:URankDistribution} \ref{prop:URankDistribution:semistrict} and Corollary \ref{cor:AbsoluteMoments} \ref{cor:AbsoluteMoments:semistrict}]
Plugging in $k=0$ to Theorem \ref{thm:AsymptoticForMoments} \ref{thm:Asymp:semistrict}, we find that
\begin{equation*}
{\rm dm}(n) \sim \frac{1}{16 n} e^{\pi \sqrt{n}}.
\end{equation*}
Note that this formula also appeared as Theorem 1.3 of \cite{BringmannMahlburg2}. We therefore have
\begin{equation*}
\frac{{\rm dm}_k(n)}{{\rm dm}(n)} \sim \frac{\log(n)^k n^{\frac{k}{2}}}{\pi^k},
\end{equation*}
and thus the normalized ratio of moments is
\begin{equation*}
\frac{{\rm dm}_k(n)}{\left(\frac{\sqrt{n} \log(n)}{\pi}\right)^k{\rm dm}(n)} \sim 1.
\end{equation*}
However, the only distribution whose moments are identically $1$ comes from the point mass probability function that satisfies $p(x = 1) = 1$, with $p(x = a) = 0$ for all $a \neq 1$.

This immediately implies both the proposition and corollary statements (for the latter, simply note that there is no difference between the absolute moments and the moments for the point mass distribution).
\end{proof}


\begin{thebibliography}{99}

\bibitem{AS64} M. Abramowitz and I. Stegun, \emph{Handbook of mathematical functions with formulas, graphs, and mathematical tables,} National Bureau of Standards Applied Mathematics Series, {\bf 55}, Washington, D.C., 1964.

\bibitem{And84} G. Andrews, \emph{Ramanujan's "lost'' notebook. IV. Stacks and alternating parity in partitions}, Adv. in Math. {\bf 53} (1984), no. 1, 55--74.

\bibitem{And12} G. Andrews, \emph{Concave and convex compositions}, Ramanujan J. {\bf 31} (2013), 67--82.

\bibitem{AndrewsGarvan1}
G. Andrews and F. Garvan,
{\it Dyson's crank of a partition},
Bull. Amer. Math. Soc. {\bf 18} (1988), 167-171.

\bibitem{AndrewsSubbaraoVidyasagar1}
G. Andrews, M. Subbarao, and M.~Vidyasagar,
\emph{A family of combinatorial identities},
Canad. Math. Bull. \textbf{15} (1972), 11--18.

\bibitem{Apo90} T. Apostol, \emph{Modular Functions and Dirichlet Series in Number Theory
Series:} Grad. Texts Math. {\bf 41},
2nd ed., 1990.

\bibitem{Billingsley1} P. Billingsley, \emph{Probability and measure. Third edition.}
Wiley Series in Probability and Mathematical Statistics. John Wiley \& Sons, Inc., New York, 1995.

%\bibitem{BringmannFolsomOnoRolen1} K. Bringmann, A. Folsom, Ken Ono, and L. Rolen,
%{\it Harmonic Maass forms and mock modular forms: theory and applications}, AMS Colloquium Series \textbf{64},
%American Mathematical Society, Providence, RI,  2017.

%\bibitem{BM14} K. Bringmann and K. Mahlburg,
%{\it Asymptotic inequalities for positive crank and rank moments},
%Trans. Amer. Math. Soc. {\bf 366} (2014), 1073--1094.

\bibitem{BringmannJenningsShafferMahlburg1}
K.~Bringmann, C.~Jennings-Shaffer, and K.~Mahlburg,
{\it On a Tauberian Theorem of Ingham and Euler-Maclaurin summation},
preprint. \texttt{arXiv:1910.03036}

\bibitem{BJSMR1} K.~Bringmann, C.~Jennings-Shaffer, K.~Mahlburg, and R.~Rhoades
{\it Peak positions of  strongly unimodal sequences},
Trans. Amer. Math. Soc., accepted for publication.

\bibitem{BringmannKim}
K.~Bringmann and B.~Kim, {\it On the asymptotic behavior of unimodal rank generating functions},
Journal of Mathematical Analysis and Applications {\bf 435} (2016), 627--645.

\bibitem{BringmannMahlburg2}
K.~Bringmann and K.~Mahlburg,
{\it Asymptotic formulas for stacks and unimodal sequences},
J. Combin. Theory Ser. A {\bf 126} (2014), 194--215.

\bibitem{BMR}
K.~Bringmann, K.~Mahlburg, and R.~Rhoades, {\it Taylor coefficients of Mock-Jacobi forms and moments of partition statistics},
Math. Proc. Cambridge Phil. Soc. \textbf{157} (2014), 231--251.

\bibitem{DesalvoPak1}
S. DeSalvo and I. Pak, \emph{Log-concavity of the partition function}, Ramanujan J. {\bf 38} (2015), 61--73.

\bibitem{DJR13} P. Diaconis, S. Janson, and R. Rhoades, \emph{Note on a partition limit theorem for rank and crank}, Bull. Lond. Math. Soc. {\bf 45} (2013), 551--553.

\bibitem{Dyson1} F. Dyson, {\it Some guesses in the theory of partitions,}
Eureka (Cambridge) {\bf 8} (1944), 10--15.

\bibitem{EL41} P. Erd\"os and J. Lehner, \emph{The distribution of the number of summands in the partitions of a positive integer}, Duke Math. J. {\bf 8} (1941), 335--345.

\bibitem{Gar} F. Garvan, \emph{New combinatorial interpretations of Ramanujan's partition congruences mod $5,7$ and $11$, } Trans. Amer. Math. Soc. {\bf 305} (1988), 47--77.

\bibitem{GORZ} M. Griffin, K. Ono, L. Rolen, and D. Zagier,
\emph{Jensen polynomials for the Riemann zeta function and other sequences},
Proceedings of the National Academy of Sciences {\bf 116} (2019), 11103--11110

%\bibitem{HardyRamanujan1}
%G.~Hardy and S.~Ramanujan, {\it Asymptotic formul\ae  for the distribution of integers of various types},
%Proc. London Math. Soc.  \textbf{16} (1917), 112--132.

\bibitem{Ingham1}
A.~Ingham,
{\it A {T}auberian theorem for partitions},
Ann. of Math. {\bf 42} (1941), 1075--1090.

\bibitem{JKB} N. Johnson, S. Kotz, and N. Balakrishnan, {\it Continuous univariate distributions. Vol. 2. Second edition.} Wiley Series in Probability and Mathematical Statistics: Applied Probability and Statistics. John Wiley and Sons, Inc., New York, 1995.

\bibitem{KimLovejoy1}
B.~Kim and J.~Lovejoy,
\emph{ The rank of a unimodal sequence and a partial theta identity of Ramanujan},
Int. J. Number Theory \textbf{10} (2014), 1081--1098.

\bibitem{KimLovejoy2}
B.~Kim and J.~Lovejoy,
\emph{ Ramanujan-type partial theta identities and rank differences for special unimodal sequences},
Ann. Comb. \textbf{19} (2015), 705--733.

\bibitem{Nic78} J. Nicolas, \emph{Sur les entiers pour lesquels il y a beaucoup de groupes ab\'{e}liens d'order $N$}, Ann. Inst. Fourier {\bf 28} (1978), 1--16.

\bibitem{Stan89} R. Stanley, \emph{Log-concave and unimodal sequences in algebra, combinatorics, and geometry}, Graph theory and its applications: East and West (Jinan, 1986), 500--535, Ann. New York Acad. Sci. {\bf 576}, New York Acad. Sci., New York, 1989.

\bibitem{Sta11} R. Stanley, \emph{Enumerative combinatorics, volume 1, second edition,} Cambridge University Press, 2011.

\bibitem{Wri71} E. Wright, \emph{Stacks. II}, Quart. J. Math. Ser. (2) {\bf 22} (1971), 107--116.

\bibitem{Zagier1}
D.~Zagier,
\emph{The Mellin transform and related analytic techniques. Appendix to E. Zeidler,
Quantum Field Theory I: Basics in Mathematics and Physics. A Bridge Between Mathematicians
and Physicists},
Springer-Verlag, Berlin-Heidelberg-New York (2006), 305--323.


\end{thebibliography}
\end{document}